\documentclass[a4paper,10pt]{article}
\usepackage[utf8]{inputenc}
\usepackage[lite]{amsrefs}
\usepackage{amssymb}
\usepackage{amsmath}
\usepackage{amsthm}
\usepackage{color}
\usepackage{pdfsync}
\usepackage{fancyhdr}
\usepackage{enumerate}
\usepackage{stmaryrd}
\usepackage{tikz-cd}
\usepackage{graphicx}
\usepackage{fullpage}

\theoremstyle{definition}
\newtheorem{definition}{Definition}[section]
\theoremstyle{definition}
\newtheorem{theorem}{Theorem}
\theoremstyle{definition}
\newtheorem*{remark}{Remark}

\newtheorem{lemma}{Lemma}

\newcommand{\N}{\mathbb{N}}

\newcommand{\Z}{\mathbb{Z}}
\newcommand{\Irr}{\textrm{Irr }}

\title{A Partial Order on Bipartitions From the Generalized Springer Correspondence}
\author{Jianqiao Xia}
\date{\today}

\begin{document}
\maketitle
\begin{abstract}
In \cite{Lusztig}, Lusztig gives an explicit formula for the bijection between the set of bipartitions and the set $\mathcal{N}$ of unipotent classes in
a spin group which carry irreducible local systems equivariant for the
spin group but not equivariant for the special orthogonal group. The set $\mathcal{N}$ has a natural partial order and therefore induces a partial order on bipartitions. We use the explicit formula given in \cite{Lusztig} to prove that this partial order on bipartitions is the same as the dominance order appeared in Dipper-James-Murphy's work (\cite{Dipper-James-Murphy}). 
\end{abstract}
\section{Preliminaries}
For group $G = \textrm{Spin}_n(k)$, where $k$ is a field of characteristic not equal to 2, let $\mathcal{N}$ be the set of unipotent classes in $G$ which carry irreducible local systems, equivariant for the conjugation action of $G$, but not equivariant for the conjugation action of the special orthogonal group. Then $\mathcal{N}$ has a one-to-one correspondence with a certain set of partitions $X_n$ (see \cite{Lusztig}, section 14).
$X_n$ consists of partitions $\lambda = (\lambda_1 \geq \lambda_2 \geq \cdots \lambda_m)$ of $n$, such that each $\lambda_i \in \N_{+}$ and 
\begin{enumerate}
\item for each integer $n \in 2\Z + 1$, the set $\{i; \lambda_i = n\}$ has at most one element;
\item for each integer $n \in 2\Z$, the set $\{i; \lambda_i = n\}$ has even number of elements. 
\end{enumerate}
Let $\Irr W_s$ be the set of all bipartitions of $s$. Then the Generalized Springer Correspondence for Spin group gives a bijection 
\begin{equation}\label{eq: 1}
X_n \longleftrightarrow \bigsqcup_{t \in 4\Z + n} \Irr W_{\frac{1}{4}(n - 2t^2 + t)}.
\end{equation}
In \cite{Lusztig}, Lusztig gives an explicit formula for this bijection. Specifically, let $\lambda = (\lambda_1 \geq \cdots \geq \lambda_m) \in X_n$. Define 
\begin{equation}
t_i = \sum_{j \geq i+1} d(\lambda_j).
\end{equation}
and 
\begin{equation}
t = \sum_{j \geq 1}d(\lambda_j).
\end{equation}
Here
\begin{equation}
d(\lambda_j) = 
\begin{cases}
0 & \textrm{If }\lambda_j \textrm{ is even.} \\
(-1)^{\frac{\lambda_j(\lambda_j - 1)}{2}} & \textrm{If }\lambda_j \textrm{ is odd.}
\end{cases}
\end{equation}
Then the image of $\lambda$ under the bijection can be constructed in the following way:
\begin{enumerate}
\item If $\lambda_i \in 4\Z + 1$, then lable this entry by $a$, and replace this entry by $\frac{1}{4}(\lambda_i - 1) - t_i$.
\item  If $\lambda_i \in 4\Z + 3$, then lable this entry by $b$, and replace this entry by $\frac{1}{4}(\lambda_i - 3) + t_i$.
\item If $\lambda_i = e \in 4\Z + 2$, then by definition it appears $2p$ times. Replace these entries by 
\begin{equation}
\frac{1}{4}(e - 2) + t_i,  \frac{1}{4}(e + 2) - t_i, \cdots, \frac{1}{4}(e + 2) - t_i 
\end{equation}
respectively, and label them as $b, a, b, \cdots, a, b, a$. 
\item If $\lambda_i = e \in 4\Z$, then by definition it appears $2p$ times. Replace these entries by 
\begin{equation}
\frac{1}{4}e + t_i,  \frac{1}{4}e - t_i, \cdots, \frac{1}{4}e - t_i 
\end{equation}
respectively. Label them as $b, a, b, \cdots, a, b, a$. 
\end{enumerate}
The modified entries with lable $a$ form an decreasing sequence $\alpha$. The entries with lable $b$ form an decreasing sequence $\beta$. If $t > 0$, then $\lambda$ corresponds to $(\alpha, \beta)$ in the bijection. If $t \leq 0$, then $\lambda$ corresponds to $(\beta, \alpha)$. Moreover, the bipartion $(\alpha, \beta)$ (when $t \geq 0$) or $(\beta, \alpha)$ (when $t \leq 0$) is an element in $\Irr W_{\frac{1}{4}(n - 2t^2 + t)}$. 
\begin{remark}
In Lusztig's paper \cite{Lusztig}, he gives the formula for partitions in increasing order. Here I simply translated everything in decreasing order, for convenience of the following proof. Moreover, for a partition in decreasing order, we can view it as an infinite sequence, by adding 0's. 
\end{remark}

There is a natural partial order on $\mathcal{N}$: $c\leq c'$ if $c$ is contained in the closure of $c'$. This partial order is given below, in terms of elements in $X_n$:
\begin{definition}
For $\lambda, \mu \in X_n$ and each is in decreasing order. We say $\lambda \leq \mu$ if and only if for all $i \in \N$
\begin{equation}
\sum_{j \leq i}\lambda_j \leq \sum_{j \leq i}\mu_j. 
\end{equation}
\end{definition}
From the bijection \eqref{eq: 1}, we have an induced partial order on the set of bipartions $\Irr W_{m}$, for each $t$. This partial order is closely related to that found in Dipper-James-Murphy's paper (\cite{Dipper-James-Murphy}), and also appears in Geck and Iancu's paper (\cite{Iancu}) as the aymptotic case for their pre-order relation on $\Irr W$, indexed by two parameters $a, b$. In the aymptotic case $b > (n - 1)a$, their pre-order is a partial order, and is defined by 
\begin{definition}{(Dipper-James-Murphy)}
The dominance order between $(\lambda, \mu), (\lambda', \mu') \in \Irr W$, each in decreasing order, is 
\begin{equation}
(\lambda, \mu) \leq (\lambda', \mu') \Leftrightarrow
\begin{cases}
\sum_{j \leq k}\lambda_j \leq \sum_{j \leq k}\lambda_j' & \textrm{for all $k$}\\
|\lambda| + \sum_{j \leq k}\mu_j \leq |\lambda'| + \sum_{j \leq k}\mu_j' & \textrm{for all $k$}.\\
\end{cases}
\end{equation}
The main result of this paper is 
\begin{theorem}
For $t \geq \frac{3}{2} m$, the induced partial order on $\Irr W_m$ from the inclusion $\Irr W_{m} \hookrightarrow X_{2t^2 - t + 4m}$, is the dominance order. 
\end{theorem}
\end{definition}
\section{Proof of Main Reult}
Let $f_{m, t}: \Irr W_{m} \hookrightarrow X_{2t^2 - t + 4m}$ be the inclusion from the Generalized Springer Correspondence. 
We first make the following observation:
\begin{lemma}
If $t \geq m$, and $\lambda \in f_{m, t}(\Irr W_m)$, then $\lambda_i \in 2\Z \cup (4\Z + 1)$.  
\end{lemma}
\begin{proof}
Suppose on the contrary there is an $i$ such that $\lambda_i \in 4\Z + 3$. By definition, $t = \sum_i d(\lambda_i)$. Each $\lambda_i \in 4\Z + 1$ contributes $+1$, and each $\lambda_i \in 4\Z + 3$ contributes $-1$. By definition of $X_n$, each odd integer appears at most once. So 
\begin{equation}
t = |\{i; \lambda_i \in 4\Z + 1\}| - |\{i; \lambda_i \in 4\Z + 3\}|. 
\end{equation}
And then $|\{i; \lambda_i \in 4\Z + 1\}| \geq t + 1$. So 
\begin{equation}
\begin{split}
2t^2 - t + 4m & = |\lambda|  = \sum_{i} \lambda_i \\
& \geq \sum_{i, \lambda_i \in 4\Z + 1} \lambda_i \\
& \geq \sum_{j = 0}^{t} (4j + 1) \\
& = 2t^2 + 3t + 1 \geq 2t^2 - t + 4m + 1.
\end{split}
\end{equation}
This is a contradiction! This lemma also proves that there are exactly $t$ odd integers in $\lambda$, each is in $4\Z + 1$. 
\end{proof}
Now the picture is clear for $t \geq m$. In fact, if $(\alpha, \beta)$ corresponds to $\lambda$, then $\alpha$ represents the deviation of odd integers of $\lambda$ from $(4t - 3, 4t - 7, \cdots, 1)$, and $\beta$ is the even integers of $\lambda$, up to scalar. We have the following lemma:
\begin{lemma}
Suppose $t \geq m$, and $(\alpha, \beta) \in \Irr W_m$ corresponds to $\lambda \in X_{2t^2 - t + 4m}$. Let $\alpha' = (\alpha'_1, \cdots, \alpha'_t)$ be the decreasing sequence of odd integers in $\lambda$, and $\beta' = (\beta'_1, \cdots, \beta'_{2k})$ be the decreasing sequence of even integers of $\lambda$. Then 
\begin{equation}
\alpha_i = \frac{1}{4}(\alpha_i' - (4(t - i) + 1))
\end{equation}
and 
\begin{equation}
\beta_i = \frac{1}{2}\beta_{2i}'.
\end{equation}
\end{lemma}
\begin{proof}
Suppose $\beta_{2s}' = e \in 2\Z$, and $4l -3 < e < 4l + 1$. We prove that $1, 5, 9, \cdots, 4l - 3$ are contained in $\alpha'$. Otherwise, there are at least $t - l + 1$ odd integers greater than $4l - 3$, and then 
\begin{equation}
\begin{split}
2t^2 - t + 4m & = |\lambda| \geq |\alpha'|  + 2e\\
& \geq 1 + 5 + \cdots + (4l - 7) + (4l + 1) + \cdots + (4t + 1) + 2 \cdot (4l - 2) \\
& = 2t^2 +3t + 4l .
\end{split}
\end{equation}
So 
\begin{equation}
t \leq m - l.
\end{equation}
This is only possible when $l = 0$.

Now suppose $e = \beta_{1}' = \beta_{2}' = \lambda_{k}$ and $l$ the same as above. Then $t_k = l$. If $l = 0$, then the lemma is automatically true. Otherwise, from Lusztig's formula, suppose there are $2p$ such $e$ in $\lambda$. There are two cases:
\begin{enumerate}
\item $e = 4l - 2$. Then those $2p$ numbers are replaced by $\frac{1}{4}(e - 2) + l, \frac{1}{4}(e + 2) - l, \cdots$ alternatively. These are exactly $2l - 1, 0, 2l - 1 \cdots, 0$. So $\beta_1 = 2l - 1 = \frac{1}{2}\beta_{2}'$. 
\item $e = 4l$. Similar as above, $\beta_1 = \frac{1}{4}e + l = 2l = \frac{1}{2}\beta_2'$. 
\end{enumerate}
In either cases, $\beta_i = \frac{1}{2}\beta_{2i}'$. 

For $\alpha$, notice that the above calculation shows that all $a$ labels from even integers gives modified number $0$. Since $\alpha$ is in decreasing order, we only need to consider  $a$ lables from odd integers. For $a$ lables from elements $\lambda_i \leq \beta_1$, we replaced it by $\frac{1}{4}(\lambda_i -1) - t_i$. Notice that the odd integers below $\beta_1$ are exactly $1, 5, 9, \cdots, 4l - 3$. So $\lambda_i$ is exactly the $(t_i + 1)$-th odd integer. They contributes to 0 in $\alpha$. 

For $\lambda_i \geq \beta_1$, there are exactly $i - 1$ odd integers greater than $\lambda_i$. So $\lambda_i = \alpha_i'$ and there are $t - i$ odd integers $\lambda_j$ with index $j$ greater than $i$. By definition, $t_i = t - i$. Therefore, 
\begin{equation}
\alpha_i = \frac{1}{4}(\lambda_i - 1) - t_i = \frac{1}{4}(\alpha_i' - 4(t - i) - 1).
\end{equation}
\end{proof}

Now we use the above observation to prove the main theorem. Let $(\alpha, \beta)$,$ (\alpha', \beta')$ be bipartitions with order $m$. They correspond to $\lambda, \lambda'$ from the inclusion $f_{m, t}: \Irr W_{m} \hookrightarrow X_{2t^2 - t + 4m}$. Here $t \geq \frac{3}{2} m$ is a fixed integer. 

\noindent \emph{Proof of main theorem:} \\
\noindent (a) If $(\alpha, \beta) \geq (\alpha', \beta')$ in the dominance order, then $\lambda \geq \lambda'$. 
\begin{proof}
Suppose $\lambda = (\lambda_1, \lambda_2, \cdots)$ in decreasing order, and $\lambda' = (\lambda_1', \lambda_2', \cdots)$ also in decreasing order.

Notice that 
\begin{equation}
\begin{split}
2t^2 - t + 4m & = \sum_{i}\lambda_i = \sum_{\lambda_i odd}\lambda_i + \sum_{\lambda_i even}\lambda_i \\
& \geq \sum_{i = 0}^{t - 1}(4i + 1) + \sum_{\lambda_i even}\lambda_i \\
& = 2t^2 - t + \sum_{\lambda_i even}\lambda_i.
\end{split}
\end{equation}
Since even numbers appear in pairs, we conclude that $\lambda_i \leq 2m$, for even entries. 

If there is an $s$ such that $4s + 1$ does not appear in $\lambda$, suppose $s$ is the smallest one. Then
\begin{equation}
2t^2 - t + 4m \geq \sum_{i = 0}^{t - 1}(4i + 1) - (4s + 1) + (4t + 1) = 2t^2 - t + 4(t - s). 
\end{equation}
So $s \geq t - m$, and $4s +1 > 2m$. This means the part that contributes to $\alpha$ and the part that contributes to $\beta$ are separated. This separation is independent of $\lambda$. In particular, we know $\lambda_{m} = \lambda_{m}'= 4(t-m) + 1$, and the odd integers after them form an arithemetic sequence with common difference 4. 

Therefore, for $k \leq m - 1$, according to lemma 2, the conditions 
\begin{equation}
\lambda_1 + \lambda_2 + \cdots + \lambda_k \geq \lambda_1' + \cdots + \lambda_k'
\end{equation}
is equilvalent to 
\begin{equation}
\alpha_1 + \cdots + \alpha_k \geq \alpha_1' + \cdots + \alpha_k'
\end{equation}
 If $\lambda \geq \lambda'$ does not hold, then there is a smallest integer $k$, such that 
\begin{equation}
\lambda_1 + \cdots + \lambda_k < \lambda_1' + \cdots + \lambda_k'.
\end{equation}
And we know from definition of $k$ that $\lambda_k < \lambda_k'$ and $k > m$. Suppose the remaining odd integers of $\lambda$ are $1, 5, 9, \cdots, 4u + 1$, and the remaining odd integers for $\lambda'$ are $1, 5, 9, \cdots, 4u' + 1$. Then $4u +2 \leq \lambda_k < \lambda_k' \leq 4u' + 5$. So $u \leq u'$. There are two cases. 

\begin{enumerate}[(1)]
\item $u' - u$ is even. The number of even integers appeared in $\{\lambda_1, \cdots, \lambda_k\}$ is $k - (t - (u+1)) = k - t + u +1$. By assumption, $k - t + u + 1$, $k - t + u' + 1$ are both odd integers or both even integers. If they are both odd, then we consider $k' = k + 1$. In this case, $\lambda_{k + 1} = \lambda_k$. So by adding one term, we still have 
\begin{equation}
\lambda_1 + \lambda_2 + \cdots + \lambda_{k'} \leq \lambda_1' + \cdots + \lambda_{k'}'.
\end{equation}
So we will only deal with the case $k - t + u + 1, k - t + u' + 1$ both even. In this case, since $|\lambda| = |\lambda'|$, we have
\begin{equation}
\sum_{i > k}\lambda_i > \sum_{i > k}\lambda_i'.
\end{equation}
So from lemma 2, 
\begin{equation}
\begin{split}
& \sum_{i = 0}^{u}(4i+1) + 4(m - |\alpha| - \beta_1 - \cdots - \beta_{\frac{k - t + u + 1}{2}}) \\
& > \sum_{i = 0}^{u'}(4i+1) + 4(m - |\alpha'| - \beta_1' - \cdots - \beta_{\frac{k - t + u' + 1}{2}}').
\end{split}
\end{equation}
Let $S(x) = |\alpha| + \beta_1 + \cdots + \beta_x$. Similarly define $S'$. Then the above can be written as 
\begin{equation}
0 > (u' - u)(2u + 2u' + 3) + 4S(\frac{k - t + u + 1}{2}) - 4S'(\frac{k - t + u' + 1}{2}).
\end{equation}
Notice that
\begin{equation}
\begin{split}
S(\frac{k - t + u + 1}{2}) - S'(\frac{k - t + u' + 1}{2})  &= S(\frac{k - t + u' + 1}{2}) - S'(\frac{k - t + u' + 1}{2}) - \sum_{i = \frac{k - t + u + 1}{2} + 1}^{\frac{k - t + u' + 1}{2}}\beta_i \\
& \geq - \sum_{i = \frac{k - t + u + 1}{2} + 1}^{\frac{k - t + u' + 1}{2}}\beta_i \\
& \geq - \frac{u' - u}{2}\frac{4u + 4}{2}. 
\end{split}
\end{equation}
The last inequality is from the definition of $u$, which implies $4u + 5$ is some $\lambda_i, i \leq k$. So all the even integer after that must be less than or equal to $4u + 4$. Therefore the corresponding $\beta$ is less than or equal to $(4u + 4)/2$. Combine these inequality, we get 
\begin{equation}
0 > (u' - u)(2u + 2u' + 3) - (u' - u)(4u + 4) = (u' - u)(2(u' - u) - 1) \geq 0.
\end{equation}
This is a contradiction, since $u' - u$ is an integer. 

\item $u' - u$ is odd. Then $u' - u \geq 1$. Let $A(x) = \lambda_1 + \cdots + \lambda_x$, and similarly define $A'(x)$. Then $B(k) := A(k) - A'(k) < 0$. Consider $B(k + 1)$. Since $\lambda_{k+1} \leq 4u + 4 < 4u' + 1 \leq \lambda_k'$, so we still have $B(k + 1) < 0$.

This process will keep going until $u' \leq u$, at some point. Notice that $u' - u$ changes by at most 1 at each step. So the first time this process ends is when $u' = u$. It must stop at some point, since both sequences $\lambda, \lambda'$ are eventually 0, and $B(x) = 0$ for large $x$. Now we can apply same method in case (1), or simply notice that in the case $u' = u$, we must have 
\begin{equation}\label{eq: 28}
\begin{split}
B(x) = A(x) - A(x)' & = 4(|\alpha| + \beta_1 + \cdots + \beta_{\frac{k - t + u + l}{2}}) \\
& - 4(|\alpha'| + \beta_1' + \cdots + \beta_{\frac{k - t + u + l}{2}}') \geq 0.
\end{split}
\end{equation}
This is a contradiction. 
\end{enumerate}

\noindent (b) Suppose $\lambda \geq \lambda'$, then $(\alpha, \beta) \geq (\alpha', \beta')$. Clearly $\alpha \geq \alpha'$ from lemma 2 and the discussion at the beginning of the proof above. So if $(\alpha, \beta) \geq (\alpha', \beta')$ does not hold, then there is a smallest $k$, such that 
\begin{equation}
|\alpha| + \beta_1 + \cdots + \beta_k < |\alpha'| + \beta_1' + \cdots + \beta_k'. 
\end{equation}
We still use the notation $A(x)$ for the sum of first $x$ terms of $\lambda$, and $S(x)$ for the sum of first $x$ terms of $\beta$ and $|\alpha|$. So $S(k) < S(k)'$, $k \geq 1$. By assumption of $k$, $\beta_k < \beta_k'$. If $\beta_k = 0$, then it is automatically a contradiction, since the left side is then $m = |\alpha'| + |\beta'|$. So $0 < \beta_k < \beta_k'$. They both comes from some even integers $2\beta_k, 2\beta_k'$ in the corresponding partition. Suppose they correspond to $\lambda_{x-1}, \lambda_x$ and $\lambda_{x' - 1}', \lambda_{x'}'$, respectively. Suppose $4u + 5 > 2\beta_k >4u + 1$ and $4u' + 5 > 2\beta_k' >4u' + 1$. Then $u' \geq u$. So $x = 2k + t - u - 1 \geq x' = 2k + t - u'-1$.
Now 
\begin{equation}
\begin{split}
A(x) - A'(x') & = A(x') - A'(x) + \sum_{i = x' + 1}^{x}\lambda_{i}' \\
& \geq \sum_{i = x' + 1}^{x}\lambda'_{i}.
\end{split}
\end{equation}
Also notice that 
\begin{equation}
\begin{split}
A(x) - A'(x') & = \sum_{i = u+1}^{u'}(4i + 1) + 4(S(k) - S'(k)) \\
& < \sum_{i = u+1}^{u'}(4i + 1).
\end{split}
\end{equation}
This means 
\begin{equation}
\sum_{i = u+1}^{u'}(4i + 1) > \sum_{i = x' + 1}^{x}\lambda'_{i}.
\end{equation}
However, this is a contradiction, since $\{4u + 5, 4u + 9, \cdots, 4u' + 1\} \subset \{\lambda_{x' + 1}, \lambda_{x' + 2}, \cdots\}$, and $x - x' = u' - u$. 
\end{proof}

We now give an example that violates the above partial order for $t = m-1$. The partition $\lambda = (4t + 1, 4t - 3, \cdots, 9, 5, 3, 1)$ corresponds to $(\alpha, \beta)$, where $\alpha = (1, 1, \cdots, 1)$ ($t$ ``1''s.) and $\beta = (1)$.

The partition $\lambda' = (4t + 1, 4t - 3, \cdots, 9, 5, 2, 2)$ corresponds to $(\alpha', \beta')$, where $\alpha' = (1, 1, \cdots, 1, 1)$ ($t + 1$ ``1''s.) and $\beta' = (0)$.

Then $\lambda > \lambda'$, but $(\alpha, \beta) < (\alpha', \beta')$. 

\section*{Acknowledgement}
The author would like to thank prof. George Lusztig for suggesting this problem and providing useful comments and directions. This project is funded by the Undergraduate Research Opportunities Program of MIT. 

\end{document}